\documentclass[12pt]{article}
 
\usepackage{amsmath, amsthm, amssymb,pdftricks}
\usepackage{mhequ}
\usepackage{mhsymb}
\usepackage{cite}
\usepackage{times}
\usepackage{algorithmic}
\usepackage{color}

\newtheorem{theorem}{Theorem}[section]
\newtheorem{corollary}[theorem]{Corollary}
\newtheorem{assumption}[theorem]{Assumption}

\theoremstyle{remark}
\newtheorem{remark}[theorem]{Remark}

\def\E{\mathbb{E}}

\definecolor{darkred}{rgb}{0.9,0.1,0.1}

\begin{document}

\title{
A patch that imparts unconditional stability \\
to certain explicit integrators for SDEs}

\author{Nawaf Bou-Rabee\thanks{Courant Institute of Mathematical
    Sciences, New York University, 251 Mercer Street, New York, NY
    10012-1185 ({\tt nawaf@cims.nyu.edu}).  N.~B-R.~was supported by
    NSF Fellowship \# DMS-0803095. }  \and Eric
  Vanden-Eijnden\thanks{Courant Institute of Mathematical Sciences,
    New York University, 251 Mercer Street, New York, NY 10012-1185
    ({\tt eve2@cims.nyu.edu}).  The work of E. V.-E. was supported in
    part by NSF grants DMS-0718172 and DMS-0708140, and ONR grant
    N00014-04-1-6046. }  }
\maketitle

\medskip

\begin{abstract}
  This paper proposes a simple strategy to simulate stochastic
  differential equations (SDE) arising in constant temperature
  molecular dynamics.  The main idea is to patch an explicit
  integrator with Metropolis accept or reject steps.  The resulting
  `Metropolized integrator' preserves the SDE's equilibrium
  distribution and is pathwise accurate on finite time intervals.  As
  a corollary the integrator can be used to estimate finite-time
  dynamical properties along an infinitely long solution.  The paper
  explains how to implement the patch (even in the presence of
  multiple-time-stepsizes and holonomic constraints), how it scales
  with system size, and how much overhead it requires.  We test the
  integrator on a Lennard-Jones cluster of particles and `dumbbells'
  at constant temperature.

\end{abstract}


\medskip




\section{Introduction}

\paragraph{Motivation}


Since Loup Verlet's landmark paper in 1967, the classical Verlet
algorithm has been the main workhorse for constant energy molecular
dynamics \cite{Ve1967}.  It is an attractive algorithm for the
Hamiltonian ODEs that arise in this context because of its
explicit, time-reversible, and symplectic nature.  In particular,
symplecticity implies long time stability of the Verlet algorithm.
The usual proof of this statement uses backward error analysis to show
that level sets of a nearby Hamiltonian function interpolate Verlet
trajectories \cite{Mo1968, BeGi1994, Re1999}.  This property implies
that a Verlet trajectory is confined to these level sets for the
duration of the simulation.  As a consequence a Verlet integrator
nearly preserves the true energy and exhibits linear growth in global
error.  Versions of Verlet to constrained (RATTLE) and multiscale
(RESPA) Hamiltonian systems are also available.  For these reasons
Verlet integrators are well-suited for long time simulation of
constant energy molecular dynamics.

The situation is quite different in the context of constant
  temperature molecular dynamics. A molecular system at constant
  temperature visits every energy isosurface with nonzero probability
  and its evolution is typically modeled using ergodic stochastic
  differential equations (SDE).  In contrast to their deterministic
counterpart, explicit integrators for SDEs diverge from the
equilibrium behavior of the true solution \cite{BoVa2010A}.  This
divergence is easy to understand since explicit integrators are only
{\em conditionally} stable.  Indeed for any time-stepsize one can find
an energy above which an explicit integrator is unstable.  As a result
stochastic effects necessarily induce instabilities by driving
trajectories to these energy values.  Since molecular simulations
often involve unbounded potential energy (e.g., Lennard-Jones
interaction), these energy values are attainable, and this issue calls
for new integration strategies for constant temperature molecular
dynamics.


\paragraph{Constant Temperature Molecular Dynamics}

We briefly recall what it means for a molecular system to be at
constant temperature. Consider $n$ molecules with masses $m_i$ for
$i=1,...,n$ evolving in a d-dimensional periodic box (or torus in $\nu
= d n$ dimensions $\mathbb{T}^{\nu}$).  Let $\boldsymbol{M}$ represent
the diagonal mass matrix of the molecular system.  We assume the
particle interaction is given by a potential energy function $U:
\mathbb{T}^{\nu} \to \mathbb{R}$.  The Hamiltonian $H:
\mathbb{T}^{\nu} \times \mathbb{R}^{\nu} \to \mathbb{R}$ of this
system can be written as: 
\begin{equation} 
  \label{Hamiltonian}
  H(\boldsymbol{q}, \boldsymbol{p}) = \frac{1}{2} \boldsymbol{p}^T
  \boldsymbol{M}^{-1} \boldsymbol{p} + U(\boldsymbol{q}) \;,
\end{equation}
where $\boldsymbol{q} \in \mathbb{T}^{\nu}$ and $\boldsymbol{p} \in
\mathbb{R}^{\nu}$ represent respectively the positions and momenta of
the molecules.  In terms of this Hamiltonian, define the probability
distribution $\mu$:
\begin{equation} 
  \label{EquilibriumDistribution}
  \mu(d \boldsymbol{q}, d \boldsymbol{p}) \eqdef 
  Z^{-1} e^{-\beta H(\boldsymbol{q}, \boldsymbol{p}) } d
  \boldsymbol{q} d \boldsymbol{p} \;
  \qquad
  Z = 
  \int_{  \mathbb{T}^{\nu}  \times \mathbb{R}^{\nu} } e^{-\beta H(\boldsymbol{q}, \boldsymbol{p}) }  d \boldsymbol{q} d \boldsymbol{p} \;.
\end{equation}
Here we have introduced the parameter $\beta$ which is inversely
related to the temperature $\mathcal{T}$ and Boltzmann constant $k_B$
via $\beta = 1/(k_B \mathcal{T})$.  

A molecular system with Hamiltonian \eqref{Hamiltonian} is at constant
temperature $\mathcal{T}$ if its trajectories sample from the
probability distribution \eqref{EquilibriumDistribution}.  The
  standard way to guarantee that this is indeed the case is to assume
  that the molecular system follows a continuous, stochastic dynamics
  of the form:
\begin{equation} \label{SDE}
\begin{cases}
  \displaystyle \frac{d \boldsymbol{Q}}{dt} &= \boldsymbol{M}^{-1} \boldsymbol{P} \;, \\
  d \boldsymbol{P} &= - \nabla U( \boldsymbol{Q} ) dt + d
  \boldsymbol{\eta}( \boldsymbol{Q}, \boldsymbol{P} )\;,
\end{cases}  
\end{equation}
where $d \boldsymbol{\eta}: \mathbb{T}^{\nu} \times \mathbb{R}^{\nu}
\to \mathbb{R}^{\nu}$ represents a thermostat force. Physically one
can interpret the thermostat force as modeling interaction between the
molecular system and a heat bath.  Mathematically it is
  essential that the thermostat force be stochastic to ensure
that the dynamics of \eqref{SDE} be ergodic with respect to
\eqref{EquilibriumDistribution}.  Several specific forms of the
  thermostat force have been proposed in the literature \cite{ScSt1978, BrBrKa1984,BuDoPa2007, SaChDe2007,LeNoTh2009} 
  and which one is best remains open for debate. This is a modeling
  question which is beyond the scope of the present paper.  Here we
  assume that the thermostat force is given, and we propose an
  integration strategy that does not make strong assumptions on its
  precise form.  In the applications and theory sections of this
paper we focus on $d \boldsymbol{\eta}$ given by Langevin dynamics
\cite{ScSt1978, BrBrKa1984} and in a companion paper \cite{BoVa2010B}
consider other thermostats including stochastic rescaling dynamics
\cite{BuDoPa2007,BuPa2008} and Nos\'{e}-Hoover-Langevin dynamics
\cite{SaChDe2007,LeNoTh2009}.  The assumption of continuity excludes,
e.g., the Andersen thermostat because it involves discrete collisions
at random times that in their wake leave the momentum of the molecular
system discontinuous \cite{An1980}.  We refer the reader to
\cite{Li2007, ELi2008} for recent progress quantifying the mixing
properties of the Andersen thermostat for molecular systems.

The SDE \eqref{SDE} is characterized by degenerate noise, irregular
drift, high-di\-men\-sio\-na\-li\-ty ($\nu$ is typically very big),
and non-well-separated time-scales.  In this context the main aim of
numerical methods is to estimate long-time dynamical properties.  This
calculation is typically done by launching a single run of an explicit
integrator and collecting statistics.  However, without the patch
introduced below this approach is prone to failure due to numerical
instabilities.

To be concrete consider computing the time-correlation in momentum
along an equilibrium path of the SDE.  This computation is common in
molecular dynamics and the reader is referred to \cite{AlTi1987,
  FrSm2002} for expository accounts.  Define the continuous
equilibrium correlation in momentum as:
\begin{equation} 
A( \tau ) = \left\langle \boldsymbol{P}(\tau + t)^T \boldsymbol{P}(t) \right\rangle \;, ~~~ (t \ge 0) \;, ~~~ (\tau \ge 0) \;,
\end{equation}
where the angle brackets denote a double average with respect to
realizations of \eqref{SDE} and an initial condition distributed
according to \eqref{EquilibriumDistribution}.  The usual way to
estimate $A(\tau)$ over a time interval $[0, T]$ is by a sample
average computed on-the-fly using a single run of an integrator.  The
numerical equilibrium correlation is defined as the limit as the
number of samples tends to $\infty$:
\begin{equation}
A^{h}(\tau) = \lim_{N \to \infty} \left( \frac{1}{N} \sum_{k=1}^{N}  
	\boldsymbol{P}_{\lfloor ( \tau + t_k ) / h \rfloor} ^T  
	\boldsymbol{P}_{\lfloor  t_k  / h \rfloor}  \right) \;.
\end{equation}
The difficulty is that this limit does not generally exist if
the integrator is explicit.  This divergence is an established problem
with explicit discretizations of SDEs that possess drifts of limited
regularity \cite{Ta2002, HiMaSt2002, MiTr2005}.


\paragraph{Proposed Integration Strategy}

This paper proposes a new integration strategy to solve the SDE
\eqref{SDE} based on combining an explicit integrator with Monte Carlo
methods to sample from the SDE's equilibrium distribution
\cite{MeRoRoTeTe1953,GoSo1989, LiSa2000}.  The resulting `Metropolized
integrator' preserves the equilibrium distribution and is often
  provably ergodic.  This feature motivates their use as sampling
methods \cite{RoTw1996A, ScLeStCaCa2006, St2007, CaLeSt2007, AkRe2008}.  In
  addition, in \cite{BoVa2010A} we showed that a Metropolized
  integrator also approximates pathwise the SDE's solution on
  finite-time intervals.

These properties ensure that a Metropolized integrator can
be used to estimate dynamics along an infinitely long solution
of \eqref{SDE}.  Indeed one can generate a long time trajectory of a
Metropolized integrator, and along any finite-time interval update
sample averages of dynamic quantities.  These averages 
converge as a consequence of ergodicity of a Metropolized
integrator.  The averages can also be made arbitrarily close to the
true solution's average by selecting the time-stepsize small
enough. For example, a Metropolized integrator can be used to
approximate to arbitrary precision the equilibrium correlation
function $A(\tau)$.  In fact, we show in this paper for every $T>0$,
there exists a $C(T)>0$ such that for $h$ sufficiently small
 \begin{equation} \label{mainresult}
 \sup_{\tau \in [0, T]} |  A^{h}(\tau) - A(h  \lfloor \tau/h \rfloor ) | \le C(T) h \;.
\end{equation}
The constant $C(T)$ increases monotonically with the length of the
time-interval $T$.  Hence, one cannot use this integrator in
situations where $T$ is very large like rare event simulation.  For
such problems the reader is referred to methods adapted to molecular
systems with rare events such as milestoning \cite{VaVe2008,
  VaVe2009}.

The error estimate \eqref{mainresult} provides a theoretical order of
accuracy of a Metropolized integrator.  However, for the strategy to
be practical, several questions remain:
\begin{itemize}
\item what does the patch involve?
\item is the patch scalable with respect to system size?  
\item does the patch work with RATTLE (Verlet with constraints) \cite{VaCi2006} or
RESPA (Verlet with multiple time-step-sizes) \cite{TuBe1991, LeMaOrWe2003}?
\end{itemize}
The aim of this paper is to answer these questions.  The paper
is organized as follows:
\begin{description}
\item[\S \ref{sec:nutshell}] provides step-by-step instructions on how to patch explicit integrators
based on Verlet, RATTLE, and RESPA, and some basic theory explaining why the patch works;
\item[\S \ref{sec:applications}] conducts numerical tests on a Lennard-Jones fluid and `dumbbell' systems;  
\item[\S \ref{sec:conclusion}] contains some conclusions and future improvements.
\end{description}


\section{Patch}   \label{sec:nutshell}

The patch involves splicing an explicit integrator with Metropolis steps.
We will show that the resulting integrator possesses the following properties: \begin{description}
\item[(P1)] preservation of the SDE's equilibrium distribution; and,
\item[(P2)] pathwise accuracy on finite time-intervals.
\end{description}
To illustrate how the patch works, we shall implement it on an explicit integrator based 
on splitting the SDE \eqref{SDE} into Hamilton's equations:  
\begin{equation} \label{HamiltonianPart}
\begin{cases}
  \displaystyle\frac{d \boldsymbol{Q}}{dt} &= \boldsymbol{M}^{-1} \boldsymbol{P}  \;, \\[6pt]
  \displaystyle\frac{d \boldsymbol{P}}{dt} &= - \nabla U(
  \boldsymbol{Q} ) \;,
\end{cases}
\end{equation}
and equations describing the effect of the thermostat:  
\begin{equation} \label{ThermostatPart}
\begin{cases}
\displaystyle\frac{d \boldsymbol{Q}}{dt} &= 0 \;, \\
d \boldsymbol{P} &=  d \boldsymbol{\eta} \;.
\end{cases}
\end{equation}
The explicit integrator considered is defined as a composition of a
step of Verlet for \eqref{HamiltonianPart} and a step of an
approximation to \eqref{ThermostatPart} (or vice versa).  If Verlet is
replaced by an implicit method (e.g., implicit Euler), then the
splitting will satisfy property (P2).  However, due to discretization
error, even an implicit integrator will generally fail to satisfy
property (P1) \cite{TaTu1990, Ta1995, Ta2002}.

Before we continue let us introduce some notation.  Let $\Omega =
\mathbb{T}^{\nu} \times \mathbb{R}^{\nu}$ denote the $2
\nu$-dimensional phase space of the molecular system and $
\boldsymbol{Y}(t) = (\boldsymbol{Q}(t), \boldsymbol{P}(t)) \in \Omega$
denote the true solution of \eqref{SDE} at time $t \ge 0$ with initial
condition $\boldsymbol{Y}(0) = \boldsymbol{x} \in \Omega$.  In what
follows we take for granted that this solution exists for all
time.


\subsection{Explicit Integrator}

Given a time-stepsize $h$ and time interval $T$, set the number of
steps to be $N = \lfloor T/ h \rfloor$ and introduce an evenly-spaced
mesh in time $t_k = h k$ for all $0 \le k \le N$.  Let $\psi_{t_{k+1},t_k}:
\Omega \to \Omega$ denote an approximation to the thermostat dynamics
\eqref{ThermostatPart}.  This map depends on time because of the
stochastic effects in the thermostat.  For example, for the
Langevin dynamics the thermostat force is given by, \[ d
\boldsymbol{\eta} = - \gamma \boldsymbol{M}^{-1} \boldsymbol{P} dt +
\sqrt{2 \gamma \beta^{-1}} d \boldsymbol{W} \;,
\]
where $\boldsymbol{W}$ is a $\nu$-dimensional Wiener process 
and $\gamma$ is a thermostat parameter \cite{ScSt1978, BrBrKa1984}. 
In this case \eqref{ThermostatPart} are Ornstein-Uhlenbeck equations in momentum whose 
pathwise unique flow is almost surely:
\begin{equation} \label{ou}
\psi_{t_{k+1}, t_k} :  (\boldsymbol{q}, \boldsymbol{p})  \mapsto  
(\boldsymbol{q}, e^{-\gamma \boldsymbol{M}^{-1} h}  \boldsymbol{p} + \boldsymbol{\eta}_{t_{k+1},t_k} )  \;,
\end{equation}
where we have introduced the random vector:
\[
\boldsymbol{\eta}_{t_{k+1},t_k} =  
\sqrt{2 \gamma \beta^{-1}}   \int_{t_k}^{t_{k+1}} 
e^{- \gamma \boldsymbol{M}^{-1} (t_{k+1} -s ) } d \boldsymbol{W}(s)  \;.
\]
The map $\psi_{t_{k+1}, t_k}$ satisfies property (P1).  For
other SDE-based thermostats an approximate map that satisfies (P1) can
be similarly constructed.  Notice that this map does not alter the
positions of the molecular system.

We introduce a second map $\hat{\theta}_h : \Omega \to \Omega$ which
approximates \eqref{HamiltonianPart}.  Since the Hamiltonian is
time-independent, this map depends only on the time-stepsize
$h=t_{k+1} - t_k$.  The patch we introduce below will require that
this map is symmetric and volume-preserving \cite{LeRe2004,
  HaLuWa2006}.  An explicit integrator for \eqref{SDE} is then given
by: \begin{equation} \label{Splitting} \hat{\phi}_{t_{k+1},t_k} =
  \psi_{t_{k+1},t_k} \circ \hat{\theta}_h \;.
\end{equation}
The order in \eqref{Splitting} does not matter since the integrator's
single step accuracy is $\mathcal{O}(h^2)$ either way.  Higher-order
accurate or implicit schemes can also be patched, but such integrators
may require more computational effort per step.  To ensure scalability
the map $\hat{\theta}_h$ in \eqref{Splitting} will use Verlet to
separately update sets of particles of the molecular system.

To this end we partition the molecular system into $m$ sets of
particles so that each set has $\nu_j$ degrees of freedom and
$\sum_{j=1}^m \nu_j = \nu$.  Given a time stepsize $h$ and initial
condition $\boldsymbol{x} \in \Omega$, a single step of
\eqref{Splitting} is defined as:
\begin{equation} \label{ExplicitIntegrator}
\hat{\boldsymbol{X}}_{1} = 
\psi_{t_1,t_0} \circ 
\underset{\hat{\theta}_h}{\underbrace{ \hat{\theta}_{h,m} \circ \cdots \circ \hat{\theta}_{h,1} }}( \boldsymbol{x} ) \;
\end{equation}
This update gives a numerical approximation to
  $\boldsymbol{Y}(h)$, $ \hat{\boldsymbol{X}}_{1} \approx
  \boldsymbol{Y}(h)$. The map $\hat{\theta}_{h,j}$ is defined as a
Verlet update of the position and momentum of the jth set of
molecules fixing the other sets.  More precisely given an input \[
\boldsymbol{x}_0 = ((\boldsymbol{q}_1, \cdots, \boldsymbol{q}_m), ( \boldsymbol{p}_1, \cdots, \boldsymbol{p}_m)) \;,
\]
where $(\boldsymbol{q}_j, \boldsymbol{p}_j) \in \mathbb{T}^{\nu_j} \times \mathbb{R}^{\nu_j}$, 
the map $\hat{\theta}_{h,j}$ outputs \[
\hat{\theta}_{h,j}(\boldsymbol{x}_0) = 
((\boldsymbol{q}_1, \cdots, \boldsymbol{q}_{j-1},\hat{\boldsymbol{q}}_{j}, \boldsymbol{q}_{j+1} , \cdots, \boldsymbol{q}_m), 
( \boldsymbol{p}_1, \cdots, \boldsymbol{p}_{j-1}, \hat{\boldsymbol{p}}_{j}, \boldsymbol{p}_{j+1}, \cdots, \boldsymbol{p}_m)) \;.
\]  Here: \begin{equation}
\begin{cases}
\hat{\boldsymbol{p}}_{j, \frac{1}{2}} &= \boldsymbol{p}_j 
- \frac{h}{2} \nabla_{Q_j} U(\boldsymbol{q}_1, \cdots, \boldsymbol{q}_m) \;, \\
\hat{\boldsymbol{q}}_{j} &= \boldsymbol{q}_j + h \boldsymbol{M}_j^{-1} \hat{\boldsymbol{p}}_{j, \frac{1}{2}} \;, \\
\hat{\boldsymbol{p}}_{j} &= \hat{\boldsymbol{p}}_{j, \frac{1}{2}}
- \frac{h}{2} \nabla_{Q_j} U(\boldsymbol{q}_1, \cdots, \boldsymbol{q}_{j-1}, 
	\hat{\boldsymbol{q}}_j, \boldsymbol{q}_{j+1}, \cdots, \boldsymbol{q}_m) \;,
	\end{cases}
\end{equation}
where $\boldsymbol{M}_j$ is the subset of the global mass matrix $\boldsymbol{M}$ 
associated to the jth set of particles.

The map $\hat{\theta}_{h,j}$ is symmetric and volume-preserving on
$\Omega$ since it is a Verlet update with respect to the Hamiltonian
\eqref{Hamiltonian} restricted to: \begin{align*} 
& \{ \boldsymbol{q}_1 \} \times \cdots \times \{ \boldsymbol{q}_{j-1} \} 
\times ( \mathbb{T}^{\nu_j} ) \times \{ \boldsymbol{q}_{j+1} \} \times \cdots \times \{ \boldsymbol{q}_{m} \} 
\times  \\
& \{ \boldsymbol{p}_1 \} \times \cdots \times \{ \boldsymbol{p}_{j-1} \} 
\times ( \mathbb{R}^{\nu_j} ) \times \{ \boldsymbol{p}_{j+1} \} \times \cdots \times \{ \boldsymbol{p}_{m} \} 
\subset \Omega \;.
\end{align*}  Even though the composite map $\hat{\theta}_h$ is a first-order splitting 
of \eqref{HamiltonianPart}, the one-step error of $\hat{\theta}_{h, j}$ in preserving energy is $O(h^3)$
because each of the separate Verlet updates is second-order accurate.


\subsection{Metropolized Integrator}

To derive an integrator that satisfies property (P1), we patch the
explicit integrator \eqref{ExplicitIntegrator} with Metropolis accept
or reject steps as follows.  Given the time-stepsize $h$,
the initial condition $\boldsymbol{x} \in \Omega$, and
the uniform random numbers $\zeta_j \sim U(0,1)$ for
$1 \le j \le m$, one step of a Metropolized Verlet integrator determines
$\boldsymbol{X}_1 \approx \boldsymbol{Y}(h)$ using:
\begin{equation} \label{mvi}
\boldsymbol{X}_{1} = 
\psi_{t_1,t_0} \circ \underset{\theta_h}{\underbrace{ \theta_{h,m} \circ \cdots \circ \theta_{h,1} }} ( \boldsymbol{x} ) \;.
\end{equation}
Here we have introduced the stochastic maps $\theta_{h,j}: \Omega \to \Omega$ 
which are Metropolized versions of the deterministic updates $\hat{\theta}_{h,j}$.
Given input \[
\boldsymbol{x}_0 = ((\boldsymbol{q}_1, \cdots, \boldsymbol{q}_m), ( \boldsymbol{p}_1, \cdots, \boldsymbol{p}_m)) \;,
 \]
the map $\theta_{h,j}$ computes a proposed move  \[
\boldsymbol{x}^{\star}_1  =
((\boldsymbol{q}_1, \cdots, \boldsymbol{q}_{j-1},\boldsymbol{q}^{\star}_{j}, \boldsymbol{q}_{j+1} , \cdots, \boldsymbol{q}_m), 
( \boldsymbol{p}_1, \cdots, \boldsymbol{p}_{j-1}, \boldsymbol{p}^{\star}_{j}, \boldsymbol{p}_{j+1}, \cdots, \boldsymbol{p}_m)) \;.
\]
using a Verlet update
\begin{equation}
\begin{cases}
\boldsymbol{p}_{j, \frac{1}{2}}^{\star} &= \boldsymbol{p}_j 
- \frac{h}{2} \nabla_{Q_j} U(\boldsymbol{q}_1, \cdots, \boldsymbol{q}_m) \;, \\
\boldsymbol{q}_{j}^{\star} &= \boldsymbol{q}_j + h \boldsymbol{M}_j^{-1} \boldsymbol{p}_{j, \frac{1}{2}}^{\star}  \;, \\
\boldsymbol{p}_{j}^{\star} &= \boldsymbol{p}_{j, \frac{1}{2}}^{\star} 
- \frac{h}{2} \nabla_{Q_j} U(\boldsymbol{q}_1, \cdots, \boldsymbol{q}_{j-1}, 
	\boldsymbol{q}_{j}^{\star} , \boldsymbol{q}_{j+1}, \cdots, \boldsymbol{q}_m) \;,
	\end{cases}
\end{equation}
and accepts this proposed move with probability \begin{equation} \label{acceptanceprobability}
1 \wedge \exp(- \beta [ H(\boldsymbol{x}^{\star}_1 ) - H( \boldsymbol{x}_0 ) ] ) \;.
\end{equation}
If this proposal is rejected the momentum is reversed.
To summarize \[
\theta_{h,j}: \boldsymbol{x}_0 \mapsto \boldsymbol{x}_1  \;,
\]
with output $\boldsymbol{x}_1$ defined as \[
\boldsymbol{x}_1 = 
((\boldsymbol{q}_1, \cdots, \boldsymbol{q}_{j-1},\bar{\boldsymbol{q}}_{j}, \boldsymbol{q}_{j+1} , \cdots, \boldsymbol{q}_m), 
( \boldsymbol{p}_1, \cdots, \boldsymbol{p}_{j-1}, \bar{\boldsymbol{p}}_{j}, \boldsymbol{p}_{j+1}, \cdots, \boldsymbol{p}_m)) \;.
\]
where
\begin{equation} \label{ithupdate}
 (\bar{\boldsymbol{q}}_{j}, \bar{\boldsymbol{p}}_{j}), = 
\begin{cases}  (\boldsymbol{q}_{j}^{\star}, \boldsymbol{p}_{j}^{\star}) & ~~\text{if}~~ \zeta_j  < 1 \wedge 
\exp(- \beta [ H(\boldsymbol{x}^{\star} ) - H( \boldsymbol{x} ) ] ) \;, \\
 (\boldsymbol{q}_{j}, -\boldsymbol{p}_{j}) & ~~ \text{otherwise} \;.
\end{cases} 
\end{equation}

The patch we propose consists of an accept or reject step like
\eqref{ithupdate}.  It requires evaluating the total energy at the
current step and at the proposed moves.  These statistics are usually
computed alongside evaluations of the force field.  When a proposed
move is rejected, the momentum of the jth set of particles is
reversed.  While these rejections ensure the integrator is
unconditionally stable and preserves the SDE's equilibrium
distribution, they cause an $O(1)$ error in accuracy due to momentum
reversals.  Next we consider the effect of these rejections and
momentum reversals on the approximation to the dynamics of the SDE.


\subsection{Quantitative Error Estimates}  \label{sec:analysis}

Consider once more the Hamiltonian of the molecular system: \begin{equation} \label{hamiltonian}
H(\boldsymbol{q}, \boldsymbol{p}) =   \frac{1}{2} \boldsymbol{p}^T \boldsymbol{M}^{-1} \boldsymbol{p} + U(\boldsymbol{q}) 
\end{equation}
where $\boldsymbol{q} \in \mathbb{T}^{\nu}$ and $\boldsymbol{p} \in
\mathbb{R}^{\nu}$ represent a configuration and momentum of the
molecular system, respectively.  For simplicity, we will assume the
thermostat dynamics is given by Langevin \cite{ScSt1978, BrBrKa1984}:
\begin{equation} \label{langevin}
\begin{cases}
\frac{d \boldsymbol{Q}}{dt} &= \boldsymbol{M}^{-1} \boldsymbol{P} \;, \\
d \boldsymbol{P} &= - \nabla U( \boldsymbol{Q} ) dt  
	- \gamma \boldsymbol{M}^{-1} \boldsymbol{P} dt + \sqrt{2 \gamma \beta^{-1}} d \boldsymbol{W} \;,
\end{cases}
\end{equation}
where $\gamma$ is a thermostat parameter, $\beta$ is the `inverse
temperature' entering the distribution~\eqref{EquilibriumDistribution}, and $\boldsymbol{W}$
is a $\nu$-dimensional Wiener process.  The theory below will rely on
the following regularity of the potential energy.

\begin{assumption} \label{pe}
The potential energy $U: \mathbb{T}^{\nu} \to \mathbb{R}$ is smooth.  
\end{assumption}

\begin{remark}
This assumption does not permit the potential force to have singularities.   For example, it holds for 
Morse potential interactions, but not Lennard-Jones interactions.  One can relax this requirement
by replacing smoothness of $U$ by some coercivity.  
\end{remark}

Let $\E^{\mu}$ denote expectation conditioned on the initial distribution 
being the equilibrium distribution of the SDE \eqref{langevin}:
\[
 \mathbb{E}^{\mu} \left( g( \boldsymbol{X}_k ) \right) = 
\int_{\Omega} 
\mathbb{E}^{\boldsymbol{x}} \left( g( \boldsymbol{X}_k ) \right) \, \mu(d \boldsymbol{x}),~~~
\boldsymbol{X}_0 = \boldsymbol{x} \in \Omega \;.
\]
The following theorem states that the Metropolized integrator satisfies property (P2).


\begin{theorem}
Assume \ref{pe}.
For every $T>0$, there exist positive constants $h_c$ and $C(T)$ 
such that,
\[
\left( \E^{\mu}  \left\{ \left|
      \boldsymbol{X}_{\lfloor t/h \rfloor}- \boldsymbol{Y}(h \lfloor t/h \rfloor)
    \right|^2 \right\} \right)^{1/2}  \le C(T) h \;,
\]
for all $h<h_c$ and $t \in [0, T]$.  
\label{MAGLAaccuracy}
\end{theorem}


A proof of this result relies on single-step accuracy and bounds on moments of the Metropolized integrator  
(see \cite{BoVa2010A}).  These bounds help to boost the single-step error estimate 
to a global error estimate, and hence, pathwise convergence on finite time-intervals.  
The Metropolized integrator initiated from equilibrium satisfies such bounds as a 
consequence of property (P1).

Theorem~\ref{MAGLAaccuracy} does not require that the Metropolized
integrator be ergodic, but only that it preserves the
equilibrium distribution.  For this reason one can extend this result
to Metropolis-adjusted discretizations of other SDE-based thermostats.
Ergodicity is technically difficult to establish when the thermostat
force is a nonlinear function of the state or highly degenerate (see,
e.g., \cite{LeNoTh2009}).  However, for Langevin dynamics
\eqref{langevin} with linear friction and additive noise on all
momenta, ergodicity is straightforward to establish for the
Metropolized integrator.

\begin{theorem}[Ergodicity] 
  Assume \ref{pe}.  For all $h>0$, observables $g: \Omega \to \mathbb{R}$, 
  and initial conditions $\boldsymbol{x} \in \Omega$,
\[
\lim_{T \to \infty}  \frac{1}{T} \int_0^T g( \boldsymbol{X}_{\lfloor t/h \rfloor}  )   dt
= 
\int_{\Omega} g(\boldsymbol{x}) \mu( d \boldsymbol{x} )  \text{.}
\]
where $\boldsymbol{X}_0 = \boldsymbol{x}$.
\label{MAGLAergodicity}
\end{theorem}

\begin{proof}
  This proof is terse.  We prove this statement for the
  Metropolized integrator based on a trivial partition of the molecular system. 
  For more details please see \cite{BoVa2010A} and references therein.  
  The Metropolized integrator by construction satisfies property (P1).  Moreover, its acceptance
  probability is strictly less than one everywhere, and the integrator
  without accept or reject steps admits a smooth transition density
  when sampled every other step.  These two observations imply that
  the smooth part of the two-step transition probability of the
  Metropolized integrator is supported everywhere, and hence, the
  chain is irreducible.  Irreducibility and property (P1) together
  imply ergodicity \cite{MenTw1996}.
\end{proof}

Theorems~\ref{MAGLAaccuracy} and \ref{MAGLAergodicity} are sufficient
to establish that the Metropolized integrator can estimate dynamics
along equilibrium trajectories of \eqref{langevin}.  For example, as a
corollary to the above theorems one can prove that the Metropolized
integrator can be used to compute equilibrium correlation
functions.

\begin{corollary}
Assume \ref{pe}.
For every $T>0$ there exist positive constants $h_c$ and $C(T)$ such that 
  \[
| A^{h}(\tau) - A(h  \lfloor \tau/h \rfloor )  | \le C(T) h \;,
\]
for all $h<h_c$ and $\tau \in [0, T]$.
\label{MAGLAcorrelationaccuracy}
\end{corollary}

A proof of this result is provided in the Appendix.


\subsection{Case of Multiple-Time-Stepsizes}

Now we present an implementation of the patch to an explicit
integrator for \eqref{SDE} based on a multiple-time-stepsize
integrator known as RESPA \cite{TuBe1991, LeMaOrWe2003}.  RESPA was
proposed for molecular systems at constant energy in
\cite{TuBeMa1992}.  This integrator is a version of Verlet adapted to
molecular systems with multiple time scales.  It is designed to
overcome a time-stepsize restriction imposed by rapidly changing
short-range interactions.  The scheme evaluates short-range forces on
a smaller time-stepsize, and hence, more frequently than long-range
forces.  The overall accuracy of the algorithm is dictated by the
largest time-stepsize.  For Hamiltonian ODEs RESPA is known to exhibit
resonance instabilities as described in \cite{FoDaLe2008}.  The
resonance occurs between forces evaluated at the coarse time-stepsize
and the normal modes of the molecular system excluding long-range
interactions.  These numerical instabilities also appear in generalizations of
RESPA to Langevin SDEs \cite{IzCaWoSk2001, MaIzSk2003, Fo2009}.  
In the following we show how to patch such a generalization to tackle 
such instabilities.

Again we partition the molecular system into $m$ sets of particles so
that each set has $\nu_i$ degrees of freedom and $\sum_{i=1}^m \nu_i =
\nu$.  We assume that the potential energy can be decomposed into fast
and slow interactions: \[ U(\boldsymbol{q}) =
U_{\text{slow}}(\boldsymbol{q}) + U_{\text{fast}}(\boldsymbol{q}) \;.
\]  
We further assume fast and slow potential forces are respectively
evaluated at $h_f$ and $h$ time increments.  The small time-stepsize
is typically a fraction of the large time-stepsize.  The integrator
\eqref{mvi} will be used, except that $\theta_{h,j}$ is replaced with
a Metropolized version of RESPA that separately updates each element
of the partition.  Given a small time-stepsize $h_f$, large
time-stepsize $h$, and input \[ 
\boldsymbol{x}_0 = ((\boldsymbol{q}_1, \cdots, \boldsymbol{q}_m), ( \boldsymbol{p}_1, \cdots, \boldsymbol{p}_m)) \;.
 \]
Set $N_f = \lfloor h/h_f \rfloor$.  The map $\theta_{h,j}$ computes a proposed move  \[
\boldsymbol{x}^{\star}_1  =
((\boldsymbol{q}_1, \cdots, \boldsymbol{q}_{j-1},\boldsymbol{q}^{\star}_{j}, \boldsymbol{q}_{j+1} , \cdots, \boldsymbol{q}_m), 
( \boldsymbol{p}_1, \cdots, \boldsymbol{p}_{j-1}, \boldsymbol{p}^{\star}_{j}, \boldsymbol{p}_{j+1}, \cdots, \boldsymbol{p}_m)) \;,
\]
using a RESPA update
\begin{equation*}
\begin{cases}
& \boldsymbol{p}_{j, \frac{1}{N_f}} = \boldsymbol{p}_j - \frac{h}{2} \nabla_{Q_j} U_{\text{slow}}(\boldsymbol{q}_1, \cdots, \boldsymbol{q}_m) \;, \\
& \boldsymbol{q}_{j, \frac{1}{N_f}} = \boldsymbol{q}_j \\
& \begin{cases}
\boldsymbol{P}_{j, \frac{k+1}{N_f}} &= \boldsymbol{p}_{j, \frac{k}{N_f}}  - \frac{h_f}{2} \nabla_{Q_j} U_{\text{fast}}(\boldsymbol{q}_1, \cdots, \boldsymbol{q}_{j-1}, 
	\boldsymbol{q}_{j,\frac{k}{N_f}} , \boldsymbol{q}_{j+1}, \cdots, \boldsymbol{q}_m) \;, \\
\boldsymbol{q}_{j, \frac{k+1}{N_f}} &= \boldsymbol{q}_{j, \frac{k}{N_f}} + h_f \boldsymbol{M}_j^{-1} \boldsymbol{P}_{j, \frac{k+1}{N_f}} \;, \\
\boldsymbol{p}_{j, \frac{k+1}{N_f}} &= \boldsymbol{P}_{j, \frac{k+1}{N_f}} 
- \frac{h_f}{2} \nabla_{Q_j} U_{\text{fast}}(\boldsymbol{q}_1, \cdots, \boldsymbol{q}_{j-1}, 
	\boldsymbol{q}_{j, \frac{k+1}{N_f}} , \boldsymbol{q}_{j+1}, \cdots, \boldsymbol{q}_m) \;, \\
\end{cases}   \\
&      \boldsymbol{q}_{j}^{\star} = \boldsymbol{q}_{j, 1} \\
&	\boldsymbol{p}_{j}^{\star} = \boldsymbol{p}_{j, 1}
- \frac{h}{2} \nabla_{Q_j} U_{\text{slow}}(\boldsymbol{q}_1, \cdots, \boldsymbol{q}_{j-1}, 
	\boldsymbol{q}_{j}^{\star} , \boldsymbol{q}_{j+1}, \cdots, \boldsymbol{q}_m) \;,
	\end{cases}
\end{equation*}
where the inner index $k$ runs from $1,\cdots,N_f-1$ (where $N_f$ is the
number of steps taken at the small time-stepsize $h_f$) and the outer
index $j$ runs from $1,\cdots,m$ (where $m$ is the number of elements in
the partition).  This proposed move is accepted with
probability \begin{equation} \label{acceptanceprobabilityrespa} 1
  \wedge \exp(- \beta [ H(\boldsymbol{x}^{\star}_1 ) - H(
  \boldsymbol{x}_0 ) ] ) \;.
\end{equation}
If this proposal is rejected the momentum is reversed as before.  This
generalization of RESPA achieves a speedup when the slow potential
force is expensive to evaluate relative to the fast potential force.


\subsection{Case of Holonomic Constraints}

Here we show how to implement the patch to molecular systems with
holonomic constraints.   Similar issues are treated in 
\cite{Ha2008, LeRoSt2010} from the viewpoint of sampling in the 
presence of constraints.  Again we partition the molecular system with
mass matrix $\boldsymbol{M}$ into $m$ sets of particles so that each
set has $\nu_i$ degrees of freedom and $\sum_{i=1}^m \nu_i = \nu$.
The main difference to the previous cases is the possibility that the
dynamics along each element of this partition is not well-defined.
For instance, if the constraint couples all atoms in the
molecule, then only the trivial partition will lead to well-posed
dynamics.  Or, if the constraint couples pairs of molecules, then only
partitions in terms of these pairs are permissible.

To rule out this possibility, we assume that the jth set of the
partition has an associated scalar constraint function $g_i :
\mathbb{T}^{\nu} \to \mathbb{R}$ independent from the positions of the
other sets.  The case of vectorial constraints can be handled quite
similarly, but for clarity we will consider scalar constraints here.
The intersection of the zero level sets of these constraint functions
defines the constraint manifold $\Sigma = \cap_{i=1}^m g_i^{-1}(0)
\subset \Omega$.  The velocities of the constrained atoms are
tangent to this manifold.  These observations motivate introducing the
set of all constrained positions and momenta denoted by $T \Sigma$
which is known as the cotangent manifold: \[ T \Sigma = \{
(\boldsymbol{q}, \boldsymbol{p}) \in \Omega ~~|~~g_i(\boldsymbol{q}) =
0,~~ \nabla_{\boldsymbol{Q}_i} g_i(\boldsymbol{q})^T
\boldsymbol{M}_i^{-1} \boldsymbol{p}_i = 0,~~ i=1,\cdots,m \} \;,
\]
where $\boldsymbol{M}_i$ is the subset of the global mass matrix
$\boldsymbol{M}$ associated to the ith set of particles.

For a molecular system with constraints, the probability distribution  
\eqref{EquilibriumDistribution} is replaced by:
\begin{equation} \label{ConstrainedEquilibriumDistribution}
d \mu_{T \Sigma}(\boldsymbol{q}, \boldsymbol{p}) \eqdef 
Z^{-1} e^{-\beta H(\boldsymbol{q}, \boldsymbol{p}) } d \sigma_{T \Sigma} (\boldsymbol{q}, \boldsymbol{p}) \;, ~~~
Z = \int_{ T \Sigma } e^{-\beta H(\boldsymbol{q}, \boldsymbol{p}) }  d \sigma_{T \Sigma} (\boldsymbol{q}, \boldsymbol{p}) \;.
\end{equation}
Here the measure $\sigma_{T \Sigma}$ represents standard volume measure on the manifold $T \Sigma$.
Introduce the Lagrange multiplier $\lambda_i(t) \in \mathbb{R}$.
The stochastic dynamics of the constrained molecular system is assumed
to be of the form:
\begin{equation} \label{ConstrainedSDE}
\begin{cases}
\displaystyle\frac{d \boldsymbol{Q}_i}{dt} &= \boldsymbol{M}_i^{-1} \boldsymbol{P}_i \;, \\
d \boldsymbol{P}_i &= - \nabla_{\boldsymbol{Q}_i} U( \boldsymbol{Q} ) dt 
+ \nabla_{\boldsymbol{Q}_i}  g_i(\boldsymbol{Q}) d \lambda_i
+ d \boldsymbol{\eta}_i( \boldsymbol{Q}, \boldsymbol{P} )\;, \\
			g_i(\boldsymbol{Q}) &= 0 \;,
\end{cases}  
\end{equation}
where $i$ enumerates the elements in the partition.  To check that \eqref{ConstrainedEquilibriumDistribution} is an 
invariant measure of \eqref{ConstrainedSDE}, its suffices to show its density is a stationary solution of the corresponding
Fokker-Planck equation.  We will assume the solution to \eqref{ConstrainedSDE} is ergodic with respect to the probability distribution 
$\mu_{T \Sigma}$.  To eliminate the Lagrange multiplier appearing in \eqref{ConstrainedSDE}, `differentiate' the constraint 
function twice along a path of $(\boldsymbol{Q}, \boldsymbol{P})$ following either the heuristic approach in 
\cite{VaCi2006} or the rigorous approach described in \cite{CiLeVa2008, BoOw2009}.

We will now introduce a constrained version of \eqref{ExplicitIntegrator}.  It is 
obtained by splitting \eqref{ConstrainedSDE} into a constrained Hamiltonian system: 
\begin{equation} \label{ConstrainedHamiltonianPart}
\begin{cases}
\displaystyle\frac{d \boldsymbol{Q}_i}{dt} &= \boldsymbol{M}_i^{-1} \boldsymbol{P}_i  \;, \\[6pt]
\displaystyle\frac{d \boldsymbol{P}_i}{dt} &= - \nabla_{\boldsymbol{Q}_i} U( \boldsymbol{Q} ) + \nabla_{\boldsymbol{Q}_i}  g_i(\boldsymbol{Q}) \lambda_{i,1}  \;, \\
			g_i(\boldsymbol{Q}) &= 0 \;,
\end{cases}
\end{equation}
and constrained thermostat dynamics:
\begin{equation} \label{ConstrainedThermostatPart}
\begin{cases}
\frac{d \boldsymbol{Q}_i}{dt} &= 0 \;, \\
d \boldsymbol{P}_i &=  d \boldsymbol{\eta}_i + \nabla_{\boldsymbol{Q}_i}  g_i(\boldsymbol{Q}) d \lambda_{i,2} \;, \\
0&=\nabla_{\boldsymbol{Q}_i} g_i(\boldsymbol{Q})^T \boldsymbol{M}_i^{-1} \boldsymbol{P}_i \;.
\end{cases}
\end{equation}
We approximate the solution to \eqref{ConstrainedHamiltonianPart} by
using a constrained version of Verlet known as RATTLE
\cite{RyCiBe1977}.  As before to maintain scalability we will
separately propagate each set of the partition using RATTLE.  
Since RATTLE moves are time-reversible and volume-preserving
\cite{LeSk1994}, a Metropolis method based on a RATTLE proposed move
and the probability distribution
\eqref{ConstrainedEquilibriumDistribution} yields an acceptance
probability that is a function of the change in energy induced by the
separate RATTLE moves.  As before we compose this map with an
approximation to \eqref{ConstrainedThermostatPart} which we denote by
$\psi_{t_{k+1},t_k}: T \Sigma \to T \Sigma$.  The step-by-step
procedure to implement this algorithm is given below.

If we assume that the thermostat in \eqref{ConstrainedThermostatPart} 
is given by Langevin dynamics,  then its pathwise unique flow is given by: 
\begin{equation} \label{constrainedou}
\psi_{t_{k+1}, t_k} : \{ (\boldsymbol{q}_i, \boldsymbol{p}_i) \}_{i=1}^{m} \mapsto  
\{ (\boldsymbol{q}_i, e^{-\gamma \mathbb{P}_i(\boldsymbol{q}) \boldsymbol{M}_i^{-1} h}  \boldsymbol{p}_i + \boldsymbol{\eta}_{t_{k+1},t_k}^i ) ) \}_{i=1}^m \;,
\end{equation}
where we have introduced the random vector:  \[
\boldsymbol{\eta}_{t_{k+1},t_k}^i =  
\sqrt{2 \gamma \beta^{-1}}   \int_{t_k}^{t_{k+1}} 
e^{- \gamma \mathbb{P}_i(\boldsymbol{q}) \boldsymbol{M}_i^{-1} (t_{k+1} -s ) } \mathbb{P}_i(\boldsymbol{q}) d \boldsymbol{W}_i(s) \;,
\] and the projection matrix: \[
 \mathbb{P}_i(\boldsymbol{q}) = \boldsymbol{I} - 
( \nabla_{\boldsymbol{Q}_i} g_i ) [ ( \nabla_{\boldsymbol{Q}_i} g_i)^T \boldsymbol{M}_i^{-1} (\nabla_{\boldsymbol{Q}_i} g_i) ]^{-1} 
(\nabla_{\boldsymbol{Q}_i} g_i)^T \boldsymbol{M}_i^{-1} \;.
\]
Here $\boldsymbol{I}$ is the $\nu_i \times \nu_i$ identity matrix.  To derive
\eqref{constrainedou} eliminate the Lagrange multiplier in
\eqref{ConstrainedThermostatPart} by differentiating the momentum
constraint and using as an integrating factor the matrix $\exp(\gamma
\mathbb{P}_i(\boldsymbol{q}) \boldsymbol{M}_i^{-1} t )$.

\begin{remark}
When the mass matrix is not the identity, \eqref{constrainedou} requires computing
the exponential of a position dependent matrix.  This calculation is non-trivial,
and can be avoided by using SHAKE in place of RATTLE and an unconstrained
Ornstein-Uhlenbeck update in place of \eqref{constrainedou}.  The resulting Metropolized
integrator will satisfy property (P1) with respect to a constrained equilibrium distribution, 
but with unconstrained velocities.   It will also satisfy property (P2).
\end{remark}

Given a time-stepsize $h$, initial condition $\boldsymbol{x} \in \Omega$, and  uniform random 
numbers $\zeta_j \sim U(0,1)$ for $j=1,\cdots,m$, one step of the Metropolized integrator 
determines $\boldsymbol{X}_1$ using \eqref{mvi}, but with the proposed moves in
$\theta_{h,j}$  obtained by RATTLE as follows.  
Introduce the discrete Lagrange multipliers $\lambda_{j,1}, \lambda_{j,2} \in \mathbb{R}$.
The integrator inputs a point in phase space on the constraint manifold \[
\boldsymbol{x}_0 = ((\boldsymbol{q}_1, \cdots, \boldsymbol{q}_m), ( \boldsymbol{p}_1, \cdots, \boldsymbol{p}_m)) \;,
\] 
and outputs a proposed move on the constraint manifold \[
\boldsymbol{x}^{\star}_1  =
((\boldsymbol{q}_1, \cdots, \boldsymbol{q}_{j-1},\boldsymbol{q}^{\star}_{j}, \boldsymbol{q}_{j+1} , \cdots, \boldsymbol{q}_m), 
( \boldsymbol{p}_1, \cdots, \boldsymbol{p}_{j-1}, \boldsymbol{p}^{\star}_{j}, \boldsymbol{p}_{j+1}, \cdots, \boldsymbol{p}_m)) \;,
\]  where $(\boldsymbol{q}_{j}^{\star}, \boldsymbol{p}_{j}^{\star})$ are determined by
\begin{equation} \label{RATTLE}
\begin{cases}
\boldsymbol{p}_{j, \frac{1}{2}}^{\star} &= \boldsymbol{p}_j 
- \frac{h}{2} \nabla_{Q_j} U(\boldsymbol{q}_1, \cdots, \boldsymbol{q}_m) 
- \frac{h}{2} \lambda_{j,1} \nabla_{Q_j}  g_j(\boldsymbol{q}_1, \cdots, \boldsymbol{q}_m) \;,  \\
\boldsymbol{q}_{j}^{\star} &= \boldsymbol{q}_j +  h \boldsymbol{M}_j^{-1}  \boldsymbol{p}_{j, \frac{1}{2}}^{\star} \;, \\
\boldsymbol{p}_{j}^{\star} &= \boldsymbol{p}_{j, \frac{1}{2}}^{\star}  
- \frac{h}{2} \nabla_{Q_j} U(\boldsymbol{q}_1, \cdots, \boldsymbol{q}_{j-1}, 
	\boldsymbol{q}_j^{\star}, \boldsymbol{q}_{j+1}, \cdots, \boldsymbol{q}_m) \\
& \qquad - \frac{h}{2} \lambda_{j,2} \nabla_{Q_j}  g_j(\boldsymbol{q}_1, ..., \boldsymbol{q}_{j-1}, 
	\boldsymbol{q}_j^{\star}, \boldsymbol{q}_{j+1}, \cdots, \boldsymbol{q}_m)
	\;, \\ 
0 &= g_j(\boldsymbol{q}_1, \cdots, \boldsymbol{q}_{j-1}, \boldsymbol{q}_j^{\star}, \boldsymbol{q}_{j+1}, \cdots, 
\boldsymbol{q}_m)  \;, \\
0&=\nabla_{\boldsymbol{Q}_j} g_j(\boldsymbol{q}_1, \cdots, \boldsymbol{q}_{j-1}, 
	\boldsymbol{q}_j^{\star}, \boldsymbol{q}_{j+1}, \cdots, \boldsymbol{q}_m)^T \boldsymbol{M}_j^{-1} \boldsymbol{p}_j^{\star} \;.
\end{cases} 
\end{equation}
The Lagrange multiplier $\lambda_{j,1}$ enforces that 
the jth proposed positions satisfy the position constraint, and $\lambda_{j,2}$ enforces
that the jth proposed velocities are tangent to the constraint manifold.
This proposed move is accepted with probability \begin{equation} \label{AcceptRejectRATTLE}
1 \wedge \exp(- \beta [ H(\boldsymbol{x}^{\star}_1 ) - H( \boldsymbol{x}_0 ) ] ) \;.
\end{equation}
This acceptance probability is a function of the change in energy induced by the RATTLE proposed move.
If this proposal is rejected the velocity is reversed, but remains tangent to the constraint manifold.


\section{Applications} \label{sec:applications}


This section tests the Metropolized integrators introduced in \S\ref{sec:nutshell}.

\subsection{Lennard-Jones Fluid}

A Lennard-Jones fluid consists of $n$ identical particles with
pairwise interactions given by a Lennard-Jones potential energy.  In
what follows we use dimensionless units to describe this system.  In
these units mass is rescaled by the mass of an individual particle (so
that the particles have unit mass), energy by the depth of the
Lennard-Jones potential energy, and length by the point where the
potential energy is zero.  We follow the notation and setup provided
in Part I of \cite{FrSm2002}.

We simulate the Lennard-Jones fluid in a fixed periodic box which we
call the simulation box.  We used a truncated version of the
  Lennard-Jones potential in which the energy between two particles a
  distance $r$ apart is kept constant after a certain cutoff distance
  and is given by:
\begin{equation} \label{LJpair}
U_{LJ}(r) = \begin{cases} f(r) - f(r_c) & r < r_c \;, \\
0 & \text{otherwise}  \;,
\end{cases}
\end{equation}
where we have introduced $f(r) = 4 ( 1/r^{12} - 1/r^{6} )$ and $r_c$
is bounded above by the size of the simulation box.  Other shifts can
be used to make the higher derivatives of $U_{LJ}$
continuous.  The error introduced by the truncation in~\eqref{LJpair} is
proportional to the density of the molecular system and can be made
arbitrarily small by selecting the cutoff distance sufficiently large.

The pair potentials are a function of the distance between the ith and
jth particle. If the position of the ith and jth particles are
$\boldsymbol{q}_i$ and $\boldsymbol{q}_j$, and the length of the
simulation box $\ell$, then this distance is given by:
\[
r_{i,j}(\boldsymbol{q}) = | (\boldsymbol{q}_i  - \boldsymbol{q}_j)\bmod \ell | \;.
\]
In terms of this pairwise distance, the potential energy of a Lennard-Jones fluid is a sum of interactions 
between all pairs of particles:
\begin{equation} \label{LJ}
U(\boldsymbol{q}) =  \sum_{i=1}^{n-1} \sum_{j=i+1}^{n} U_{LJ}(r_{i,j}( \boldsymbol{q}) ) \;.
\end{equation}
Evaluating the potential force requires $O(n^2)$ operations (where
$\nu$ is the dimension of configuration space), and typically
dominates total computation cost.


\subsection{Autocorrelation of Lennard-Jones Fluid}

Here we test the accuracy of the Metropolized integrator \eqref{mvi}
based on trivial and per particle partitions of the molecular system.
In the former proposed moves in the Metropolis steps are obtained by
per particle Verlet updates, and in the latter by global Verlet
updates.  The numerics indicate that the order of accuracy of both
methods is roughly $O(h^{2})$ and that the error constant of the per
particle partition is approximately an order of magnitude smaller.

To test accuracy we use the Metropolized integrator to estimate the
equilibrium momentum autocorrelation of the Lennard-Jones fluid.  For
the numerical experiment, we set the fluid's density $\rho=0.8442$ and
temperature $\mathcal{T}=0.728$ (units are dimensionless as described
earlier).  For these values the phase of the Lennard-Jones fluid is
liquid, and close to the triple (gas-liquid-solid) point.  We also fix
the number of particles to be $n=25$ and the degrees of freedom to be
$d=2$.  The size of the simulation box is $\ell = (n/\rho)^{1/2}
\approx 5.04$.  A reasonable cutoff distance in \eqref{LJpair} at the
selected density is $r_c = 2.5$.  The initial positions of the
particles are chosen to be the vertices of a square lattice that fills
the simulation box.  For instance, the length of each square in the
lattice can be chosen to be $\ell/\lceil n^{1/2} \rceil$.  The initial
velocities are sampled from the Maxwell distribution. The thermostat
parameter is set equal to $\gamma = 1.0$.

In the simulations we estimate the true velocity autocorrelation over a time-interval $[0,1]$.
Set $N=\lfloor T / h \rfloor$ and introduce an evenly-spaced mesh in computational 
time $t_k = h k$ for all $0 \le k \le N$.   Let  $A^{h}: [0, 1] \to \mathbb{R}$ denote the velocity
correlation function obtained by the Metropolized integrator \eqref{mvi}: \[
A^{h}(\tau) = \lim_{N \to \infty}
\frac{1}{N} \sum_{k=1}^{N} \boldsymbol{P}_{\lfloor ( \tau + t_k ) / h \rfloor} ^T  
\boldsymbol{P}_{\lfloor ( t_k ) / h \rfloor}  \;, ~~~ (\tau \ge 0) \;.
\]
Define the relative Richardson error as \begin{equation} \label{richardsonerror}
\epsilon_h \eqdef \frac{\sup_{\tau  \in [0,1]} | A^{h}( \tau)  - A^{2h}(\tau) |}{ \sup_{\tau \in [0,1]} | A(\tau) |} \;.
\end{equation} An empirical estimate of $\epsilon_h$ is plotted for time-stepsizes \[
h=\{0.005, 0.0025, 0.00125, 0.000625\}
\] in Fig.~\ref{fig:AutocorrelationLennardJonesFluid} with $N=10^8$.  The denominator in $\epsilon_h$
is calculated by using the approximation of $A$ at $h=0.000625$.
The figure shows that the error of the Metropolized 
integrator is approximately $O(h^{2})$.  Moreover it shows that the Metropolized integrator based on a 
per particle partition is nearly an order of magnitude more accurate than the 
algorithm based on a trivial partition.


\subsection{Scaling of Metropolized Integrator}

Until now the numerics dealt with a fixed number of particles in $d=2$ dimensions.  Next we 
address how the integrator \eqref{mvi} scales with the number of particles keeping the density and 
temperature fixed.   As before we consider two types of partitions: per particle and trivial.  
This time we will assume the particles are in three-dimensional space.  Recall 
that in the former proposed moves in the Metropolis steps are obtained by per particle 
Verlet updates, and in the latter by global Verlet updates.  We will show that the type
of proposed move affects the scalability of the Metropolized integrator.  In particular, the per particle
partition will lead to a scalable algorithm.

By fixing the density and temperature, the stiffness of the molecular
system is fixed.  Thus, from the viewpoint of numerical analysis, the
time-stepsize ought to be independent of system size.  However, the
acceptance probability in the Metropolized integrator depends on the
change in energy induced by the proposal move (see
\eqref{acceptanceprobability}).  If this proposal move is global, then
the magnitude of this change in energy increases with system size.
Hence, the acceptance probability is inversely related to system size,
in general.  This poor scaling of Metropolis methods based on global
moves is well-known in the literature (see, e.g., \cite{RoRo1998,
  BePiRoSaSt2010, MaPiSt2010}).

For the numerical experiment, we set the fluid's density $\rho=0.8442$
and temperature $\mathcal{T}=0.728$ (as before units are
dimensionless).  The mean acceptance probability is computed along a
long time trajectory of the integrators ($10^6$ steps) with a fixed
time-stepsize of $h=0.01$ and a variety of system sizes.  The initial
positions of the particles are the vertices of a cubic lattice
contained in the simulation box.  The length of each cube is given by
$\lceil (1/\rho)^{1/3} \rceil$ ($\rho=0.8442$ is the density of the
fluid).  The initial velocities are sampled from the Maxwell
distribution.  The thermostat parameter and Lennard-Jones cutoff
distances are set equal to $\gamma = 1.0$ and $r_c =2.5$,
respectively.

Figure~\ref{fig:Scaling} shows the outcome of the experiment: the mean
acceptance probability per particle for the per particle and trivial
partitions as a function of the number of particles.  The acceptance
probability for the trivial partition clearly deteriorates with system
size.  On the other hand, the acceptance probability for the per
particle partition is independent of system size.  In fact, the
acceptance probability per particle is equal to one to within round
off error.  This result seems to defy intuition since the particles
experience Lennard-Jones interaction.  But, keep in mind that these
interactions are, in fact, short-range due to the cutoff in the
Lennard-Jones potential energy.  Hence, the change in energy induced
by per particle Verlet moves is independent of system size.


\subsection{Autocorrelation of Lennard-Jones Dumbbells}

A rigid dumbbell is a type of molecule that involves a holonomic constraint.
It consists of a pair of particles constrained to a fixed distance from one another.  
The constraint arises when the spring joining a flexible dumbbell 
infinitely stiffens.   Consider $n$ identical dumbbells where particles in
separate dumbbells interact via a Lennard-Jones potential energy \eqref{LJ}.

For the numerical experiment, we simulate the dumbbells in a
simulation box of length $\ell$ and $d=2$ dimensions as before.  We
set the system's density $\rho=0.998$ and temperature
$\mathcal{T}=3.0$ (units are dimensionless as described earlier).  We
also fix the number of dumbbells to $n=30$ and the length of each
dumbbell to $\ell_0 = 1$.  If the positions of the ith pair of
particles describing the ith dumbbell are $\boldsymbol{q}_{i,1}$ and
$\boldsymbol{q}_{i,2}$, then the constraint function associated to the
ith dumbbell is given by:
\[
g_i(\boldsymbol{q}) = | (\boldsymbol{q}_{i,1}  - \boldsymbol{q}_{i,2})\bmod \ell |^2  - \ell_0^2, ~~~i=1,\cdots,n \;.
\]
The size of the simulation box is $\ell = (2 n/\rho)^{1/2} \approx
7$. The cutoff distance in \eqref{LJpair} is set at $r_c=3.0$.  The
initial positions of the dumbbells are chosen randomly in the
simulation box, but with no overlap. The initial velocities are
sampled from the Maxwell distribution constrained to be tangent to the
constraint manifold at the initial positions and temperature. The
thermostat parameter is set to $\gamma = 1.0$.

To quantify the accuracy of the numerical method, we use the
Metropolized integrator based on a per dumbbell partition.  We again
use the relative Richardson error defined in \eqref{richardsonerror}
to estimate the rate of convergence of the method.
Figure~\ref{fig:AutocorrelationLennardJonesDumbbells} shows the
velocity autocorrelation that we wish to approximate over the
time-interval $[0,1]$.  It is noticeably different from the velocity
autocorrelation in the case of the Lennard-Jones cluster without
constraints.  The figure also shows a plot of the Richardson error as
a function of time-stepsize.  The rate of convergence appears to be
approximately $O(h^{2})$.


\begin{figure}[ht!]
\begin{center}
\includegraphics[scale=0.7,angle=0]{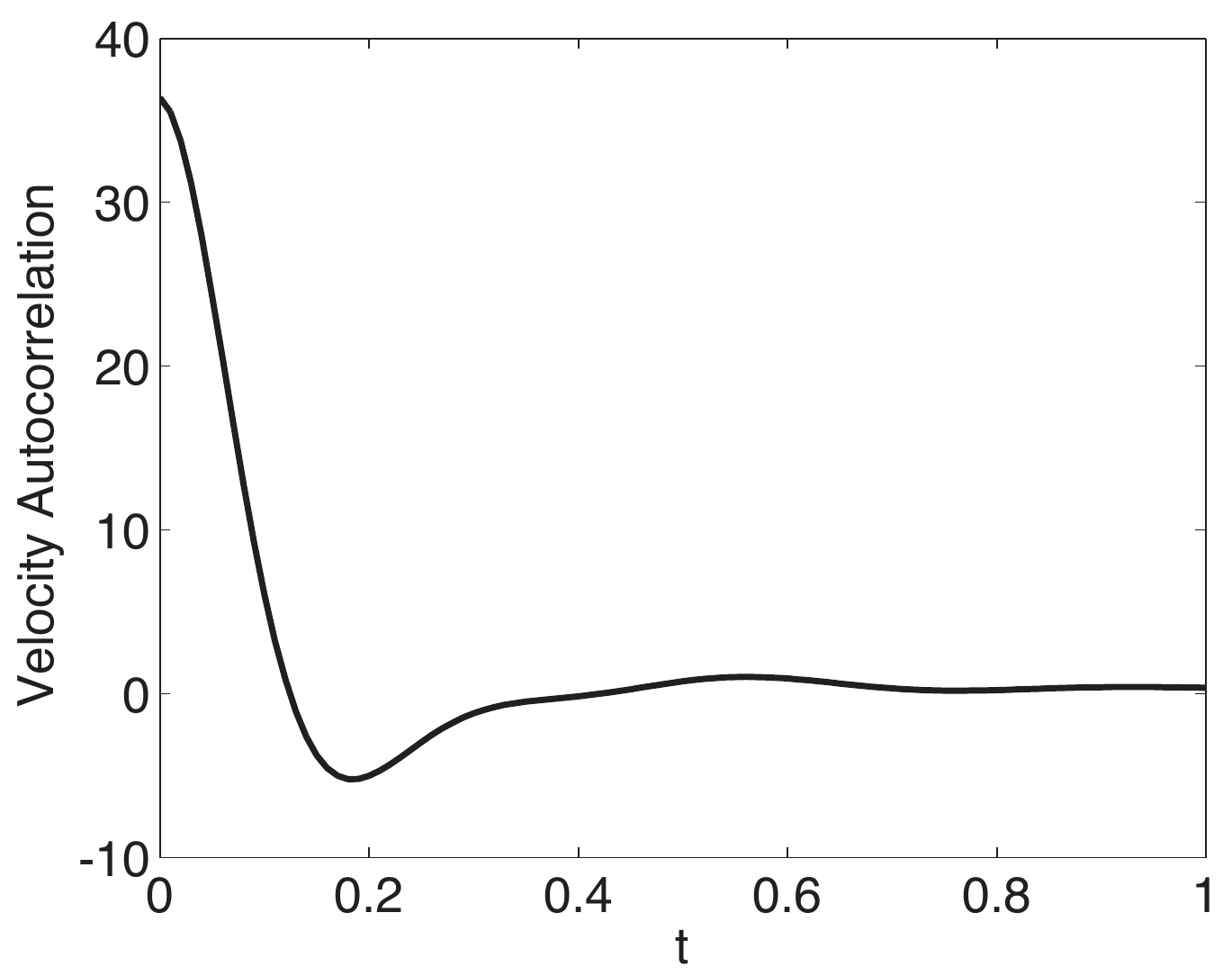}
\includegraphics[scale=0.7,angle=0]{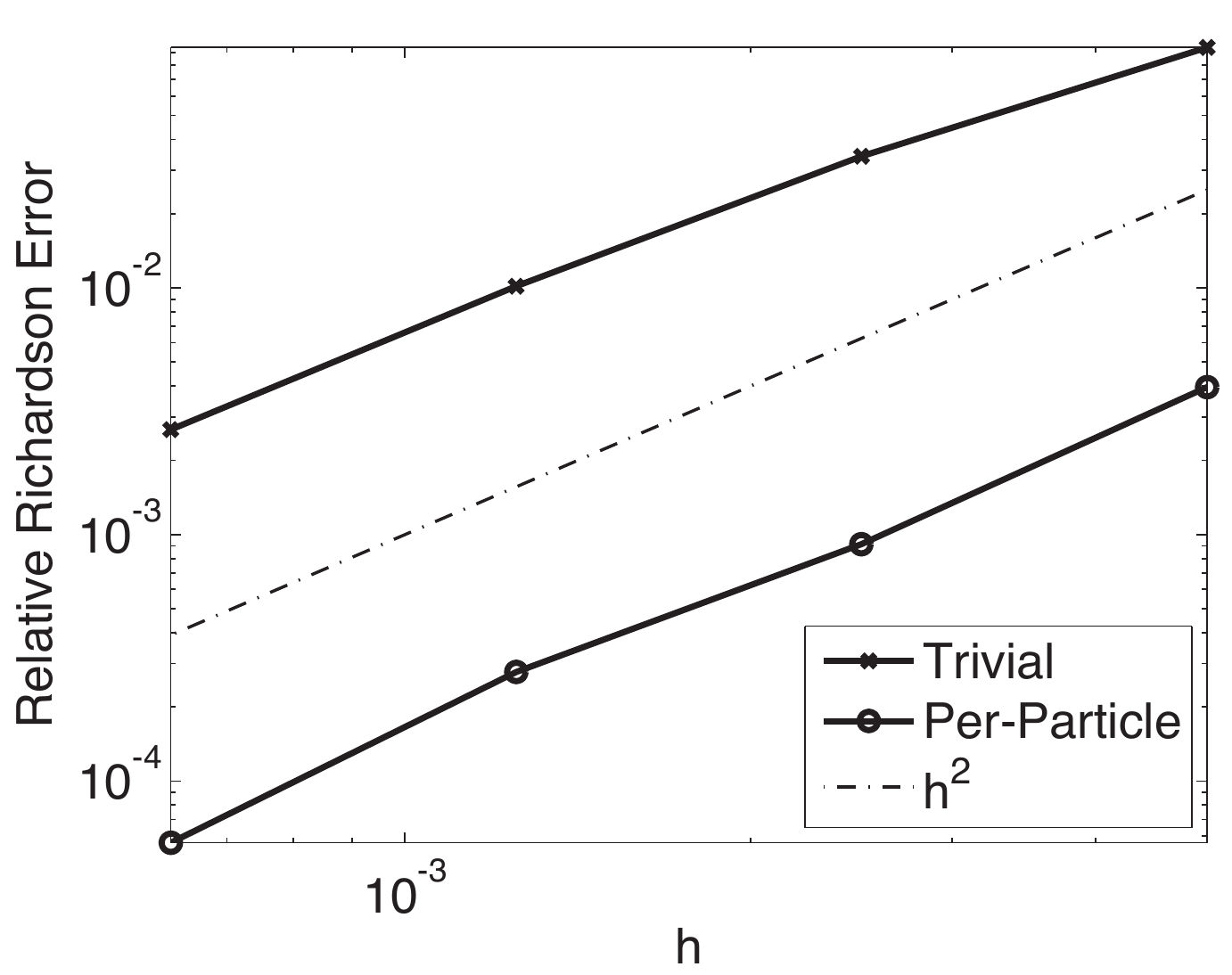}
\caption{ \small {\bf Autocorrelation of Lennard-Jones Fluid.}  The
  panel on the top shows the velocity autocorrelation function over
  the time interval $[0,1]$.  The panel on the bottom shows a loglog
  plot of the Richardson error (\eqref{richardsonerror}) of the
  Metropolized integrator based on trivial and per-particle
  partitions.  The molecular system integrated is the Lennard-Jones
  fluid with $25$ particles in a square box. The dashed line
  represents a reference slope of $h^2$.  A total of $N=10^8$ samples
  were used to obtain these empirical estimates.  Notice that the
  Metropolized integrator based on per-particle Verlet moves is about
  an order of magnitude more accurate than the Metropolized integrator
  based on global Verlet moves.  }
\label{fig:AutocorrelationLennardJonesFluid} 
\end{center}
\end{figure}

\begin{figure}[ht!]
\begin{center}
\includegraphics[scale=0.8,angle=0]{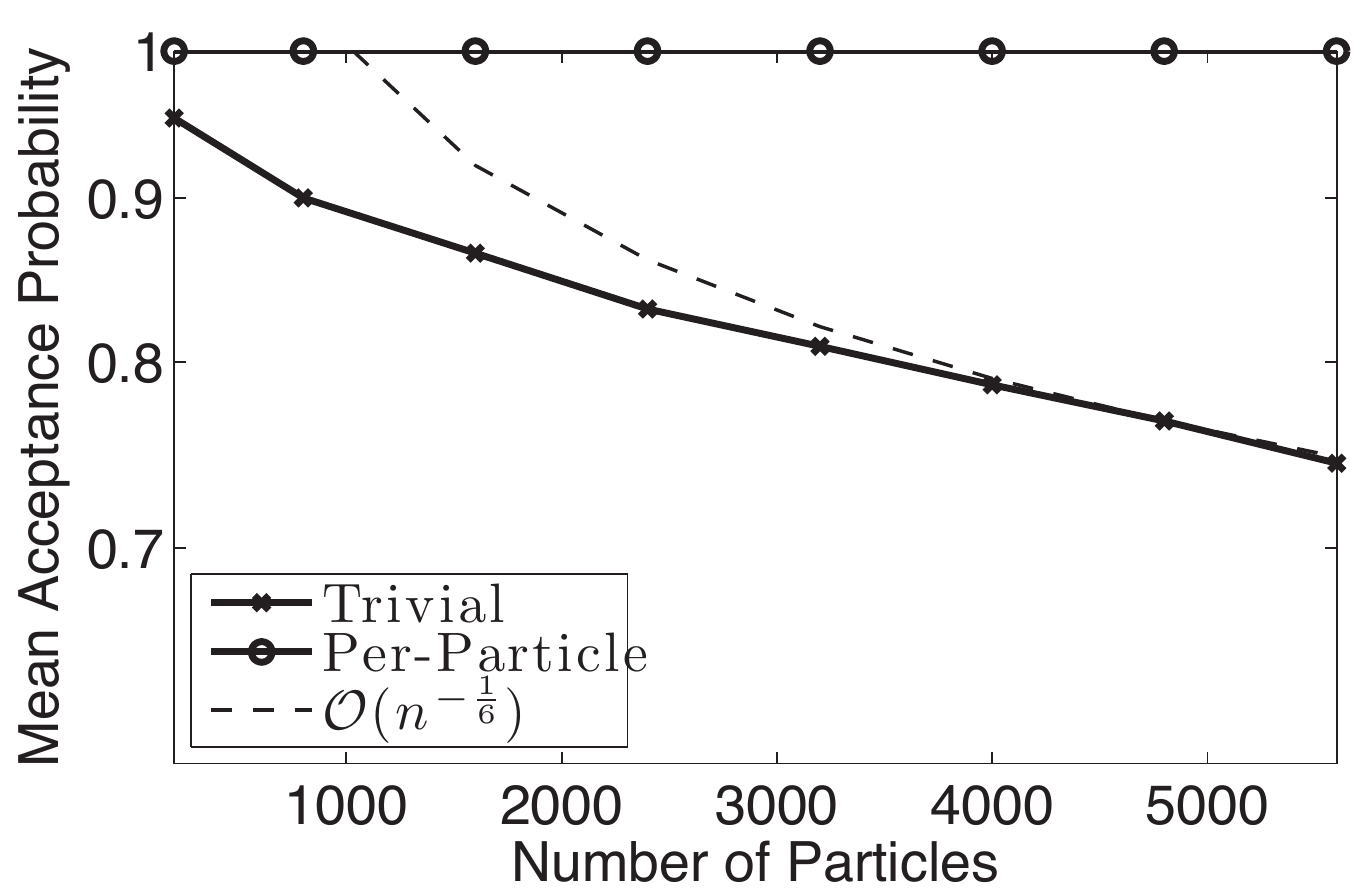}
\caption{ \small {\bf Mean Acceptance Probability as a Function of
    System Size.}  This figure shows the mean acceptance probability
  per particle for the integrator \eqref{mvi} based on a trivial and
  per particle partition.  From the plot it is clear that a) the
  acceptance probability for the integrator based on a per particle
  partition is independent of system size and b) the acceptance
  probability for the trivial partition scales like $n^{-1/6}$ where
  $n$ is the number of particles.  }
\label{fig:Scaling} \end{center}
\end{figure}

\begin{figure}[ht!]
\begin{center}
\includegraphics[scale=0.7,angle=0]{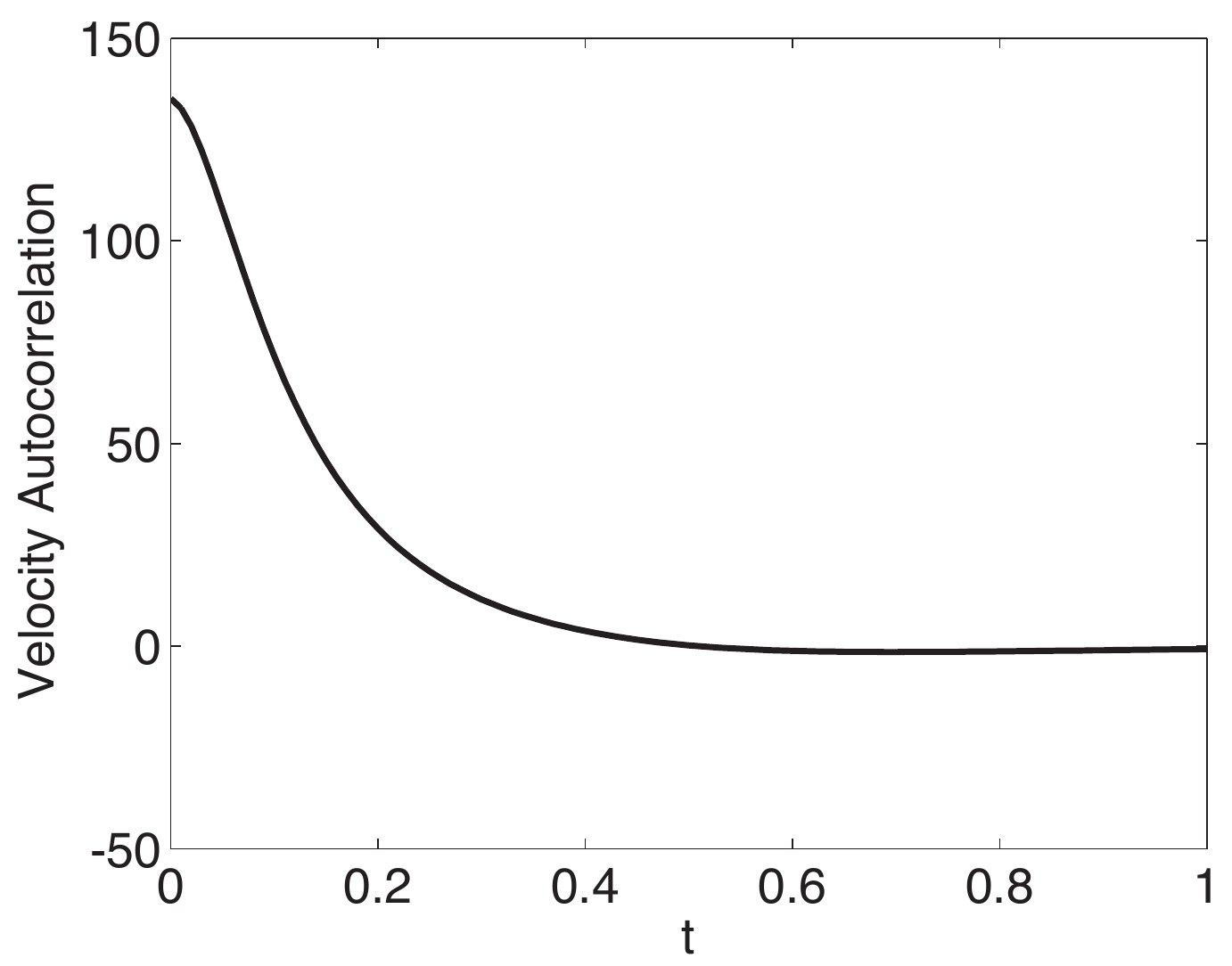}
\includegraphics[scale=0.7,angle=0]{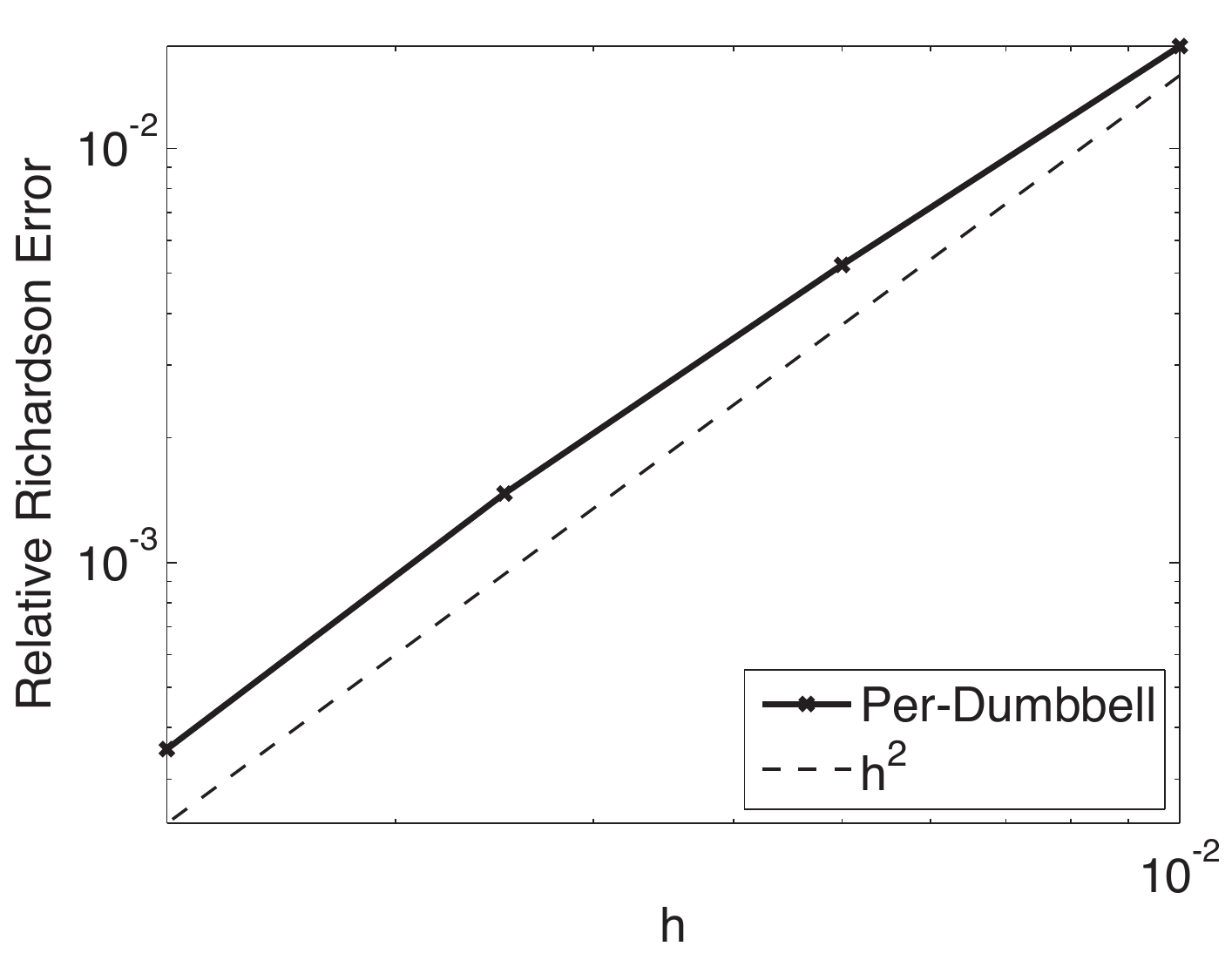}
\caption{ \small {\bf Autocorrelation of Lennard-Jones Dumbbells.}
  The panel on the top shows the velocity autocorrelation function.
  The panel on the bottom shows a loglog plot of an empirical estimate
  of the Richardson error (\eqref{richardsonerror}) of the
  Metropolized integrator based on per-dumbbell RATTLE proposed moves.
  The molecular system integrated is the Lennard-Jones system of $15$
  dumbbells in a square box. The dashed line represents a reference
  slope of $h^{2}$.  A total of $N=10^8$ samples were used to obtain
  this estimate.  }
\label{fig:AutocorrelationLennardJonesDumbbells} 
\end{center}
\end{figure}


\section{Conclusions} \label{sec:conclusion}

This paper showed how thinking probabilistically helps to design good
integrators for SDEs arising in molecular dynamics.  The paper brought
ideas from Monte Carlo into molecular dynamics, to obtain a
Metropolized integrator that is unconditionally stable, pathwise
accurate, and still explicit.  While the examples and theory in the
paper focused on Langevin dynamics, the methodology can be extended to
other thermostats.  The patch we propose is simple and, as we
  showed, its computational overhead is minimum compared to an
  `unpatched' integrator.

An open question about the Metropolized integrator is its convergence
rate to equilibrium.  In general, it is not expected that the
Metropolized integrator inherits all of the mixing properties of the
exact solution to the SDE. The main reason is conditional stability of
the underlying Verlet integrator that generates proposal moves.
Indeed for any time-stepsize one can find an energy value above which
the drift in this integrator gives proposed moves that increase the
energy, in contrast to the exact drift in the SDE which always centers
the solution towards lower energy values.  Since higher energy values
have a lower equilibrium probability weight, these proposed moves are
typically rejected.  While these rejections ensure that the
Metropolized integrator is ergodic, at high energy values they prevent
the integrator from inheriting all of the mixing properties of the
true solution.

This issue has been addressed for the MALA algorithm in the overdamped
limit of Langevin dynamics \cite{BoHaVa2010}.  It turns out that MALA
converges to its equilibrium distribution at an exponential rate up to
terms exponentially small in time-stepsize.  However, in the
overdamped limit positions are no longer differentiable, momentum is
no longer present, and the MALA algorithm does not involve momentum
flips.  Intuitively one expects these momentum flips to reduce
stagnation at high energy values, and hence, enhance the mixing rate.
Future research will investigate the effect of these momentum flips to
this mixing rate.

Another open question concerns the acceptance probability of the
Metropolized integrator as the dimension tends to $\infty$.  The
numerical experiment in Figure~\ref{fig:Scaling} indicates that the
acceptance probability of the Metropolized integrator based on global
Verlet moves scales as $\mathcal{O}(n^{-1/6})$ where $n$ is the number
of particles.  Thus, to obtain an $\mathcal{O}(1)$ acceptance
probability the time stepsize should be proportional to $n^{-1/6}$,
and the integrator would require $\mathcal{O}(n^{1/6})$ steps to
traverse phase space.  This scaling property would imply that the
Metropolized integrator is more efficient at making $\mathcal{O}(1)$
moves in state space as compared with the $\mathcal{O}(n^{1/4})$ and
$\mathcal{O}(n^{1/3})$ scalings of respectively hybrid Monte-Carlo and
MALA \cite{BePiRoSaSt2010}.  This question will also be investigated in future work.

\paragraph{Acknowledgements}

We wish to thank Gerard Ben-Arous,  Weinan E, Martin Hairer, 
Mikael Rechtsman, Christof Sch\"{u}tte, Andrew Stuart, Maddalena Venturoli, 
and Jonathan Weare for stimulating discussions. 
The research of NBR was supported in part by NSF Fellowship DMS-0803095. 
The research of EVE was supported in part by NSF grants DMS-0718172 and DMS-0708140, 
and ONR grant N00014-04-1-6046.

\appendix
\section*{Appendix}

\begin{proof}[Proof of Corollary~\ref{MAGLAcorrelationaccuracy}]
We prove this estimate for the equilibrium correlation of any Lipschitz function
 $g \in C^{0,1}(\Omega, \mathbb{R})$.  This assumption 
 includes the special case of a scalar velocity autocorrelation function.  Ergodicity of the Metropolized integrator 
 (see Theorem~\ref{MAGLAergodicity}) implies that
\[
\lim_{N \to \infty}  \frac{1}{N} \sum_{k=1}^{N} g(\boldsymbol{X}_{\lfloor t_k/h \rfloor}) g(\boldsymbol{X}_{\lfloor (t_k + \tau)/h \rfloor} ) 
= \left\langle g(\boldsymbol{X}_{\lfloor t/h \rfloor}) g(\boldsymbol{X}_{\lfloor (t + \tau)/h \rfloor} )  \right\rangle
\]
for any $t \ge 0$.  (Recall, that the angle brackets denote a double average with respect to initial conditions distributed 
according to the equlibrium distribution of \eqref{langevin} and realizations of the Wiener process.)
Thus, we wish to estimate:
\begin{align*}
\epsilon_h &\eqdef  \left|  \left\langle g(\boldsymbol{X}_{\lfloor t/h \rfloor}) g(\boldsymbol{X}_{\lfloor (t + \tau)/h \rfloor} )  \right\rangle 
-  \left\langle g(\boldsymbol{Y}(t)) g(\boldsymbol{Y}(t + h \lfloor \tau/ h \rfloor ))  \right\rangle  \right|  \\
& =  \left|  \left\langle g(\boldsymbol{x}) \left( g(\boldsymbol{X}_{\lfloor \tau/h \rfloor} ) - g(\boldsymbol{Y}( h \lfloor \tau/ h \rfloor )  \right)  ) \right\rangle  \right|  \;.
\end{align*} 
By the Cauchy-Schwarz inequality,
\begin{align*}
&\epsilon_h  \le  C_g 
\left(\int_{\Omega} 
\left| \E^{\boldsymbol{x}} \left\{ g(\boldsymbol{X}_{\lfloor \tau/h \rfloor} ) - g(\boldsymbol{Y}( h \lfloor \tau/ h \rfloor))  \right\} \right|^2 d \mu  \right)^{1/2} \;,
\end{align*}
where $C_g =  ( \int_{\Omega} | g( \boldsymbol{x} ) |^2 d \mu )^{1/2}$.
By Jensen's inequality,
\begin{align*}
&\epsilon_h  \le  C_g  
\left(   \E^{\mu}\left\{ \left| g(\boldsymbol{X}_{\lfloor \tau/h \rfloor}) -g(\boldsymbol{Y}(h \lfloor \tau/ h \rfloor) ) \right|^2 \right\}  \right)^{1/2} \;.
\end{align*}
The Lipschitz assumption on $g$ implies,
\[
| g( \boldsymbol{x} ) - g( \boldsymbol{y} ) | \le L_g | \boldsymbol{x} - \boldsymbol{y} |  
\]
for some constant $L_g > 0$ and for every $\boldsymbol{x}, \boldsymbol{y} \in \Omega$.  
Hence,
\begin{align*}
\epsilon_h  \le  L_g C_g
\left(   \E^{\mu}\left\{ \left| \boldsymbol{X}_{\lfloor \tau/h \rfloor} -\boldsymbol{Y}(h \lfloor \tau/ h \rfloor)  \right|^2 \right\}  \right)^{1/2} \;.
\end{align*}
Pathwise convergence of the Metropolized integrator (see Theorem~\ref{MAGLAaccuracy}) implies
\begin{align*}
\left(   \E^{\mu}\left\{ \left| \boldsymbol{X}_{\lfloor \tau/h \rfloor} -\boldsymbol{Y}(h \lfloor \tau/ h \rfloor)  \right|^2 \right\}  \right)^{1/2} \le C(T) h \;.
\end{align*}
From which it follows, that the accuracy of the Metropolized integrator in computing the equilibrium correlation in $g$ 
is $O(h)$ with a prefactor that depends on the function $g$. 
\end{proof}

\bibliographystyle{amsplain}
\bibliography{nawaf}

\providecommand{\bysame}{\leavevmode\hbox to3em{\hrulefill}\thinspace}
\providecommand{\MR}{\relax\ifhmode\unskip\space\fi MR }
\providecommand{\MRhref}[2]{%
  \href{http://www.ams.org/mathscinet-getitem?mr=#1}{#2}
}
\providecommand{\href}[2]{#2}
\begin{thebibliography}{10}

\bibitem{AkRe2008}
E.~Akhmatskaya and S.~Reich, \emph{{GSHMC}: An efficient method for molecular
  simulation}, J. Comput. Phys. \textbf{227} (2008), 4937--4954.

\bibitem{AlTi1987}
M.~P. Allen and D.~J. Tildesley, \emph{Computer simulation of liquids},
  Clarendon Press, 1987.

\bibitem{An1980}
H.~C. Andersen, \emph{Molecular dynamics simulations at constant pressure
  and/or temperature}, J. Chem. Phys. \textbf{72} (1980), 2384.

\bibitem{BeGi1994}
G.~Benetin and A.~Giorgilli, \emph{On the {H}amiltonian interpolation of near
  to the identity symplectic mappings with applications to symplectic
  integration algorithms}, J.~Statist.~Phys.~ \textbf{74} (1994), 1117--1143.

\bibitem{BePiRoSaSt2010}
A.~Beskos, N.~S. Pillai, G.~O. Roberts, J.~M. Sanz-Serna, and A.~M. Stuart,
  \emph{Optimal tuning of hybrid {M}onte-{C}arlo}, arXiv:1001.4460, 2010.

\bibitem{BoHaVa2010}
N.~B{ou-Rabee}, M.~H{airer}, and E.~V{anden-Eijnden}, \emph{Non-asymptotic
  mixing of the {MALA} algorithm}, arXiv:1008.3514v1 [math.PR], 2010.

\bibitem{BoOw2009}
N.~Bou-Rabee and H.~Owhadi, \emph{Stochastic variational integrators}, IMA J.
  of Numer. Anal. \textbf{29} (2009), 421--443.

\bibitem{BoVa2010B}
N.~Bou-Rabee and E.~Vanden-Eijnden, \emph{Reconciling {M}onte {C}arlo and
  molecular dynamics}, Preprint, 2009.

\bibitem{BoVa2010A}
N.~B{ou-Rabee} and E.~V{anden-Eijnden}, \emph{Pathwise accuracy and ergodicity
  of {M}etropolized integrators for {SDE}s}, CPAM \textbf{63} (2010), 655--696.

\bibitem{BrBrKa1984}
A.~Br\"{u}nger, C.~L. Brooks, and M.~Karplus, \emph{Stochastic boundary
  conditions for molecular dynamics simulations of {ST2} water}, Chem. Phys.
  Lett. \textbf{105} (1984), 495--500.

\bibitem{BuDoPa2007}
G.~Bussi, D.~Donadio, and M.~Parrinello, \emph{Canonical sampling through
  velocity rescaling}, J. Chem. Phys. \textbf{126} (2007), 014101.

\bibitem{BuPa2008}
G.~Bussi and M.~Parrinello, \emph{Stochastic thermostats: Comparison of local
  and global schemes}, Computer Physics Communications \textbf{179} (2008),
  26--29.

\bibitem{CaLeSt2007}
E.~Canc\'{e}s, F.~Legoll, and G.~Stoltz, \emph{Theoretical and numerical
  comparison of some sampling methods for molecular dynamics}, Mathematical
  Modelling and Numerical Analysis \textbf{41} (2007), 351--389.

\bibitem{CiLeVa2008}
G.~Ciccotti, T.~Lelievre, and E.~Vanden-Eijnden, \emph{Projections of
  diffusions on submanifolds: Application to mean force computation}, CPAM
  \textbf{61} (2008), 0001--0039.

\bibitem{ELi2008}
W.~E and D.~Li, \emph{The {A}ndersen thermostat in molecular dynamics}, CPAM
  \textbf{61} (2008), 96--136.

\bibitem{Fo2009}
W.~Fong, \emph{Multi-scale methods in time and space for particle simulations},
  Ph.D. thesis, Stanford University, 2009.

\bibitem{FoDaLe2008}
W.~Fong, E.~Darve, and A.~Lew, \emph{Stability of asynchronous variational
  integrators}, J. Comput. Phys. \textbf{227} (2008), 8367--8394.

\bibitem{FrSm2002}
D.~Frenkel and B.~Smit, \emph{Understanding molecular simulation: From
  algorithms to applications, second edition}, Academic Press, 2002.

\bibitem{GoSo1989}
J.~Goodman and A.~Sokal, \emph{Multigrid {M}onte {C}arlo method: Conceptual
  foundations}, Phys. Rev. D \textbf{40} (1989), 2035--2071.

\bibitem{HaLuWa2006}
E.~Hairer, C.~Lubich, and G.~Wanner, \emph{Geometric numerical integration},
  Springer Series in Computational Mathematics, vol.~31, Springer, 2006.

\bibitem{Ha2008}
C.~Hartmann, \emph{An ergodic sampling scheme for constrained {H}amiltonian
  systems with applications to molecular dynamics}, J.~Stat.~Phys. \textbf{130}
  (2008), 687--712.

\bibitem{HiMaSt2002}
D.~J. Higham, X.~Mao, and A.~M. Stuart, \emph{Strong convergence of
  {E}uler-type methods for nonlinear stochastic differential equations}, SIAM
  J. Numer. Anal. \textbf{40} (2002), 1041--1063.

\bibitem{IzCaWoSk2001}
J.~A. Izaguirre, D.~P. Catarello, J.~M. Wozniak, and R.~D. Skeel,
  \emph{Langevin stabilization of molecular dynamics}, J. Chem. Phys.
  \textbf{114} (2001), 2090.

\bibitem{LeNoTh2009}
B.~Leimkuhler, E.~Noorizadeh, and F.~Theil, \emph{A gentle stochastic
  thermostat for molecular dynamics}, J. Stat. Phys. \textbf{135} (2009),
  261--277.

\bibitem{LeRe2004}
B.~Leimkuhler and S.~Reich, \emph{Simulating {H}amiltonian dynamics}, Cambridge
  Monographs on Applied and Computational Mathematics, Cambridge University
  Press, 2004.

\bibitem{LeSk1994}
B.~Leimkuhler and R.~Skeel, \emph{Symplectic numerical integrators in
  constrained {H}amiltonian systems}, JCP \textbf{112} (1994), 117--125.

\bibitem{LeRoSt2010}
T.~Leli\`{e}vre, M.~Rousset, and G.~Stoltz, \emph{Langevin dynamics with
  constraints and computation of free energy differencs}, arXiv:1006.4914v1,
  2010.

\bibitem{LeMaOrWe2003}
A.~Lew, J.~E. Marsden, M.~Ortiz, and M.~West, \emph{Asynchronous variational
  integrators}, Arch. Ration. Mech. An. \textbf{167} (2003), 85--145.

\bibitem{Li2007}
D.~Li, \emph{On the rate of convergence to equilibrium of the {A}ndersen
  thermostat in molecular dynamics}, J. Stat. Phys. \textbf{129} (2007),
  265--287.

\bibitem{LiSa2000}
J.~S. Liu and Y.~N. Wu, \emph{Generalised {G}ibbs sampler and multigrid {M}onte
  {C}arlo for {B}ayesian computation}, Biometrika \textbf{87} (2000), 353--369.

\bibitem{MaIzSk2003}
Q.~Ma, J.~A. Izaguirre, and R.~D. Skeel, \emph{{VERLET-I/R-RESPA/IMPULSE} is
  limited by nonlinear instabilities}, SIAM J. Sci. Comput. \textbf{24} (2003),
  1951--1973.

\bibitem{MaPiSt2010}
J.~C. Mattingly, N.~S. Pillai, and A.~M. Stuart, \emph{Diffusion limits of the
  random walk {M}etropolis algorithm in high dimensions}, arXiv:1003.4306,
  2010.

\bibitem{MenTw1996}
K.~L. Mengersen and R.~L. Tweedie, \emph{Rates of convergence of the {H}astings
  and {M}etropolis algorithms}, Ann. Stat. \textbf{24} (1996), 101--121.

\bibitem{MeRoRoTeTe1953}
N.~Metropolis, A.~W. Rosenbluth, A.~H. Teller, and E.~Teller, \emph{Equations
  of state calculations by fast computing machines}, J. Chem. Phys. \textbf{21}
  (1953), 1087--1092.

\bibitem{MiTr2005}
G.~N. Milstein and M.~V. Tretyakov, \emph{Numerical integration of stochastic
  differential equations with nonglobally {L}ipschitz coefficients}, SIAM J.
  Numer. Anal. \textbf{43} (2005), 1139--1154.

\bibitem{Mo1968}
J.~Moser, \emph{Lectures on {H}amiltonian systems}, vol.~81, Mem.~AMS, 1968.

\bibitem{Re1999}
S.~Reich, \emph{Backward error analysis for numerical integrators}, SIAM J.
  Num. Anal. \textbf{36} (1999), 1549--1570.

\bibitem{RoRo1998}
G.~O. Roberts and J.~S. Rosenthal, \emph{Optimal scaling of discrete
  approximations to {L}angevin diffusions}, J. Roy. Statist. Soc. Ser. B
  \textbf{60} (1998), 255--268.

\bibitem{RoTw1996A}
G.~O. Roberts and R.~L. Tweedie, \emph{Geometric convergence and central limit
  theorems for multidimensional {H}astings and {M}etropolis algorithms},
  Biometrika \textbf{1} (1996), 95--110.

\bibitem{RyCiBe1977}
J.~Ryckaert, G.~Ciccotti, and H.~Berendsen, \emph{Numerical integration of the
  {C}artesian equations of motion of a system with constraints: Molecular
  dynamics of n-alkanes}, JCP \textbf{23} (1977), 327--341.

\bibitem{SaChDe2007}
A.~A. Samoletov, M.~A. Chaplain, and C.~P. Dettmann, \emph{Thermostats for
  ``slow'' configurational modes}, J. Stat. Phys. \textbf{128} (2008),
  1321--1336.

\bibitem{ScLeStCaCa2006}
A.~Scemama, T.~Leli\`{e}vre, G.~Stoltz, E.~Canc\'{e}s, and M.~Caffarel,
  \emph{An efficient sampling algorithm for variational {M}onte {C}arlo}, J.
  Chem. Phys. \textbf{125} (2006), 114105.

\bibitem{ScSt1978}
T.~Schneider and E.~Stoll, \emph{Molecular-dynamics study of a
  three-dimensional one-component model for distortive phase transitions},
  Physical Review B \textbf{17} (1978), 1302--1322.

\bibitem{St2007}
G.~Stoltz, \emph{Some mathematical methods for molecular and multiscale
  simulation}, Ph.D. thesis, Ecole Nationale des Ponts et Chauss\'{e}es, 2007.

\bibitem{Ta1995}
D.~Talay, \emph{Simulation and numerical analysis of stochastic differential
  systems : a review}, Probabilistic Methods in Applied Physics (P.~Kr\`{e}e
  and W.~Wedig, eds.), vol. 451, Springer-Verlag, Berlin, 1995, pp.~54--96.

\bibitem{Ta2002}
\bysame, \emph{Stochastic {H}amiltonian systems: Exponential convergence to the
  invariant measure, and discretization by the implicit {E}uler scheme}, Markov
  Processes and Related Fields \textbf{8} (2002), 1--36.

\bibitem{TaTu1990}
D.~Talay and L.~Tubaro, \emph{Expansion of the global error for numerical
  schemes solving stochastic differential equations}, Stoch. Anal. Appl.
  \textbf{8} (1990), 94--120.

\bibitem{TuBe1991}
M.~E. Tuckerman and B.~J. Berne, \emph{Stochastic molecular dynamics in systems
  with multiple time scales and memory friction}, J. Chem. Phys. \textbf{95}
  (1991), 4389--4396.

\bibitem{TuBeMa1992}
M.~E. Tuckerman, B.~J. Berne, and G.~Martyna, \emph{Reversible multiple time
  scale molecular dynamics}, J. Chem. Phys. \textbf{97} (1992), 1990--2001.

\bibitem{VaCi2006}
E.~Vanden-Eijnden and G.~Ciccotti, \emph{Second-order integrators for
  {L}angevin equations with holonomic constraints}, Chem. Phys. Lett.
  \textbf{429} (2006), 310--316.

\bibitem{VaVe2008}
E.~Vanden-Eijnden and M.~Venturoli, \emph{Markovian milestoning with {V}oronoi
  tessellations}, J. Chem. Phys. \textbf{130} (2008), 194101.

\bibitem{VaVe2009}
\bysame, \emph{Exact rate calculations by trajectory parallelization and
  twisting}, In press, 2009.

\bibitem{Ve1967}
L.~Verlet, \emph{Computer ``experiments'' on classical fluids. {I.}
  {T}hermodynamical properties of {L}ennard-{J}ones molecules}, Phys.~Rev.
  \textbf{159} (1967), 98--103.

\end{thebibliography}

 \end{document}